\documentclass{amsart}
\usepackage{amsmath,amssymb,amsthm}
\usepackage[longnamesfirst,numbers]{natbib}
\usepackage[margin]{fixme}
\usepackage{setspace}
\usepackage{enumerate}         

\newtheorem{thm}{Theorem}[section]
\newtheorem{cor}[thm]{Corollary}
\newtheorem{lem}[thm]{Lemma}
\newtheorem{prop}[thm]{Proposition}
\theoremstyle{definition}
\newtheorem{defn}[thm]{Definition}

\theoremstyle{remark}
\newtheorem{rem}[thm]{Remark}

\numberwithin{equation}{section}

\newcommand{\norm}[1]{\left\Vert#1\right\Vert}
\newcommand{\abs}[1]{\left\vert#1\right\vert}
\newcommand{\set}[1]{\left\{#1\right\}}
\newcommand{\Ind}[1]{\mathbf{1}_{\left\{#1\right\}}}
\newcommand{\RR}{\mathbb{R}}

\newcommand{\PP}{\mathbb{P}}
\newcommand{\CC}{\mathbb{C}}
\newcommand{\NN}{\mathbb{N}}
\newcommand{\FF}{\mathbb{F}}

\newcommand{\bbS}{\mathbb{S}}
\newcommand{\cU}{\mathcal{U}}

\newcommand{\cF}{\mathcal{F}}

\newcommand{\cQ}{\mathcal{Q}}

\newcommand{\Rplus}{\mathbb{R}_{\geqslant 0}}

\newcommand{\pd}[2]{\frac{\partial #1}{\partial #2}}

\newcommand{\scal}[2]{\left\langle{#1},{#2}\right\rangle}

\renewcommand{\Re}{\mathrm{Re}}

\newcommand{\cD}{\mathcal{D}}

\newcommand{\supp}{\operatorname{supp}}
\newcommand{\wt}[1]{{\widetilde{#1}}}
\newcommand{\wh}[1]{{\widehat{#1}}}
\newcommand{\E}[1]{\mathbb{E}\left[#1\right]}                           
\newcommand{\Ex}[2]{\mathbb{E}^{#1}\left[#2\right]}                     
\newcommand{\Excond}[3]{\mathbb{E}^{#1}\left[\left.#2\right|#3\right]}  

\title{Regularity of affine processes on general\linebreak state spaces}
\author{Martin Keller-Ressel}
\address{Department of Mathematics, TU Berlin, Germany}
\email{mkeller@math.tu-berlin.de}

\author{Walter Schachermayer}
\address{Faculty of Mathematics, University of Vienna, Austria}
\email{walter.schachermayer@univie.ac.at}

\author{Josef Teichmann}
\address{Department of Mathematics, ETH Zurich, Switzerland}
\email{jteichma@math.ethz.ch}

\thanks{The first and third author gratefully acknowledge the support by the ETH foundation. The second author gratefully acknowledges financial support from the Austrian Science Fund (FWF) under grant P19456, from the European Research Council (ERC) under grant FA506041 and from the Vienna Science and Technology Fund (WWTF) under grant MA09-003. Furthermore this work was financially supported by the Christian Doppler Research Association (CDG).\\
The authors would like to thank Enno Veerman and Maurizio Barbato for comments on an earlier draft}

\keywords{affine process, regularity, semimartingale, generalized Riccati equation}

\subjclass[2000]{60J25}

\date{\today}

\begin{document}
\maketitle

\begin{abstract}We consider a stochastically continuous, affine Markov process in the sense of \citet{Duffie2003}, with
c\`adl\`ag paths, on a general state space $D$, i.e.~an arbitrary Borel subset of $\RR^d$. We show that such a process is always
regular, meaning that its Fourier-Laplace transform is differentiable in time, with derivatives that are continuous in the
transform variable. As a consequence, we show that generalized Riccati equations and L\'evy-Khintchine parameters for the process
can be derived, as in the case of $D = \Rplus^m \times \RR^n$ studied in \citet{Duffie2003}. Moreover, we show that when the
killing rate is zero, the affine process is a semi-martingale with absolutely continuous characteristics up to its time of explosion.
Our results generalize the results of \citet{KST2011} for the state space $\Rplus^m \times \RR^n$ and provide a new probabilistic
approach to regularity.
\end{abstract}

\section{Introduction}
A time-homogeneous, stochastically continuous Markov process $X$
on the state space $D \subset \RR^d$ is called affine, if its
transition kernel $p_t(x,d\xi)$ has the following property:  There
exist functions $\Phi$ and $\psi$, taking values in $\CC$ and $\CC^d$ respectively, such that
\begin{equation*}
\int_D e^{\scal{\xi}{u}}p_t(x,d\xi) = \Phi(t,u)
\exp(\scal{x}{\psi(t,u)})
\end{equation*}
for all $t \in \Rplus$, $x \in D$ and $u$ in the set $\cU = \set{u \in \CC^d: \sup_{x \in D} \Re \scal{u}{x} < \infty}$.

The class of stochastic processes resulting from this definition is
a rich class that includes Brownian motion, L\'evy processes,
squared Bessel processes, continuous-state branching processes with
and without immigration \citep{Kawazu1971}, Ornstein-Uhlenbeck-type
processes \citep[Ch.~17]{Sato1999}, Wishart processes \citep{Bru1991} and several models from
mathematical finance, such as the affine term structure models of
interest rates \citep{Duffie1996} and the affine stochastic
volatility models \citep{Kallsen2006} for stock prices.

For a state space of the form $D = \Rplus^m \times \RR^n$ the class
of affine processes has been originally defined and systematically
studied by \citet{Duffie2003}, under a regularity condition. In this
context, regularity means that the time-derivatives
\begin{align*}
F(u) = \frac{\partial \Phi(t,u)}{\partial t}\Bigg |_{t=0+}, \qquad
R(u) = \frac{\partial \psi(t,u)}{\partial t}\Bigg |_{t=0+}
\end{align*}
exist for all $u \in \cU$ and are continuous on subsets $\cU_k$ of $\cU$ that exhaust $\cU$. Once
regularity is established, the process $X$ can be described
completely in terms of the functions $F$ and $R$. The problem of showing that regularity of a stochastically continuous affine process $X$ always holds true was originally considered for processes on the state space $D = \Rplus^m \times \RR^n$, and was proven -- giving a positive answer -- by
\citet{KST2011}, building on results by \citet{Dawson2006}
and \citet{K2008b}.

Already \citet{Duffie2003} remarked that affine processes can be
considered on other state spaces $D \neq \Rplus^m \times \RR^n$,
where also no reduction to the `canonical' case by embedding or linear transformation is possible. One such
example is given by the Wishart process (for $d \geq 2$), which is an affine process taking values in $S^+_d$, the cone of positive semidefinite $d
\times d$-matrices. Recently, \citet{Cuchiero2009} gave a full characterization of all affine processes with state space $S_d^+$
and \citet{CKMT2011} consider the even more general case, when $D$ is an `irreducible symmetric cone' in the sense of \citet{Faraut1994}, which
includes the $S_d^+$ case.\footnote{A symmetric cone is a self-dual convex cone $D$, such that for any two points $x,y \in D$ a linear
automorphism $f$ of $D$ exists, which maps $x$ into $y$. It is called irreducible if it cannot be written as a non-trivial direct sum of two
other symmetric cones.} 

In both articles, regularity of the process remains a crucial ingredient,
and the authors give direct proofs showing that regularity follows
from the definition of the process, as in the case of $D = \Rplus^m
\times \RR^n$. 
Even though the affine processes on $\Rplus^m \times
\RR^n$ and on symmetric cones are regular and have been
completely classified, it is known that this does not amount to a full classification of all affine processes on a general state space
$D$. A simple example is given by the process $X_t^{(x,x^2)}=(B_t + x, (B_t +
x)^2)_{t \geq 0}$, where $B$ is a standard Brownian motion. This
process is an affine process that lives on the parabola $D =
\set{(y,y^2), y \in \RR} \subset \RR^2$, and can be characterized by the functions \[\Phi(t,u) = \frac{1}{\sqrt{1 - 2t u_2}}
\exp\left(\frac{u_1^2 t}{2(1 - 2t u_2)}\right),\quad \psi(t,u) = (u_1,u_2)/(1 - 2t u_2).\] It can even be extended into an affine process
on the parabola's epigraph \linebreak $\set{(y,z): z \ge y^2, y\in \RR}$ (see \citet[Sec.~12.2]{Duffie2003}), but not into a process on the state
space $\Rplus^m \times \RR^n$, or on any symmetric cone. For more general results in this direction we refer to \citet{Spreij2010},
who provide a classification of affine diffusion processes on polyhedral cones and state spaces which are level sets of quadratic functions
(`quadratic state spaces'). They start from a slightly different definition of an affine process through a
stochastic differential equation, which also immediately implies the regularity of the process.

The contribution of this article is to show that on any state space $D$, the regularity of an affine process follows from the exponentially affine form of the characteristic function under the assumption of c\`adl\`ag paths. So far, most of the proofs given in the literature have
used certain properties of the state space: In the case of $S_d^+$ and the symmetric cones the fact that the set $\cU$ has open
interior, and in the case $\Rplus^m \times \RR^n$ a degeneracy argument that reduces the problem to $\Rplus^m$, which is again a
symmetric cone. The existence of an non-empty interior of $\cU$ leads to a purely analytical proof based in broad terms on the theory of
differentiable transformation semigroups of \citet{Montgomery1955}; see \citet{KST2011}. In general, we cannot guarantee that $\cU$ has
non-empty interior, and the analytic technique ceases to work. Therefore, we now use a substantially different -- probabilistic --
technique that is independent of the nature of the state space under consideration.
A different proof of regularity of affine processes on general state spaces has been obtained in work parallel to this article in \citet{Cuchiero2011} (see also \citet{Cuchiero2011a}).

\section{Definitions and Preliminaries}

Let $D$ be a non-empty Borel subset of the real Euclidian space $\RR^d$, equipped with the Borel $\sigma$-algebra $\cD$, and assume
that the affine hull of $D$ is the full space $\RR^d$. To $D$ we add a
point $\delta$ that serves as a `cemetery state', define
\[\wh{D} = D \cup \set{\delta}, \qquad \wh{\cD} = \sigma(\cD, \set{\delta}),\]
and equip $\wh{D}$ with the Alexandrov topology, in which any open
set with a compact complement in $D$ is declared an open
neighborhood of $\delta$.\footnote{Note that the topology of $\wh{\cD}$ enters our assumptions in a subtle way: We require later that $X$ is c\`adl\`ag on $\wh{\cD}$, which is a property for which the topology matters.} Any continuous function $f$ defined on $D$
is tacitly extended
to $\wh{D}$ by setting $f(\delta) = 0$.\\
Let $(\Omega, \cF, \FF)$ be a filtered space, on which a family
$(\PP^x)_{x \in \wh{D}}$ of probability measures is defined, and
assume that $\cF$ is $\PP^x$-complete for all $x \in \wh{D}$ and
that $\FF$ is right continuous. Finally let $X$ be a c\`adl\`ag process
taking values in $\wh{D}$, whose transition kernel
\begin{equation}
 p_t(x,A) = \PP^x (X_t \in A), \qquad (t \ge 0, x \in \wh{D}, A \in \wh{\cD})
\end{equation}
is a normal time-homogeneous Markov kernel, for which $\delta$ is
absorbing. That is, $p_t(x,.)$ satisfies the following:
\begin{enumerate}[(a)]
 \item $x \mapsto p_t(x,A)$ is $\wh{\cD}$-measurable for each $(t,A) \in \Rplus \times \wh{\cD}$. 
 \item $p_0(x,\set{x}) = 1$ for all $x \in \wh{D}$,
 \item $p_t(\delta,\set{\delta}) = 1$ for all $t \geq 0$
 \item $ p_t(x,\wh{D}) = 1 $ for all $(t,x) \in \Rplus \times
 \wh{D}$, and
 \item the Chapman-Kolmogorov equation
 \[p_{t+s}(x,d\xi) = \int p_t(y,d\xi) \,p_s(x,dy) \]
  holds for each $t,s \ge 0$ and $(x,d\xi) \in \wh{D} \times \wh{\cD}$.
\end{enumerate}

We equip $\RR^d$ with the canonical inner product $\scal{}{}$, and associate to $D$ the set $\cU \subseteq \CC^d$ defined by
\begin{equation}
 \cU = \set{u \in \CC^d: \sup_{x \in D} \Re \scal{u}{x} < \infty}.
\end{equation}
Note that the set $\cU$ is the set of complex vectors $u$ such that the exponential function $x \mapsto e^{\scal{u}{x}}$ is bounded on $D$. It is easy to see that $\cU$ is a convex cone and always contains the set of purely imaginary vectors $i\RR^d$. We will also need the sets
\begin{equation}
 \cU_k = \set{u \in \CC^d: \sup_{x \in D} \Re \scal{u}{x} \le k}, \qquad k \in \NN,
\end{equation}
for which we note that $\cU = \bigcup_{k \in \NN} \cU_k$.

\begin{defn}[Affine Process]\label{Def:affine_process}
The process $X$ is called \emph{affine} with state space $D$, if its transition kernel $p_t(x,d\xi)$ satisfies the following:
\begin{enumerate}[(i)]
 \item it is stochastically continuous, i.e.~$\lim_{s \to t}p_s(x,.) = p_t(x,.)$ weakly for all $t \ge 0, x \in D$, and
 \item its Fourier-Laplace transform depends on the initial state in the following way: there exist functions $\Phi: \Rplus \times \cU \to \CC$ and $\psi: \Rplus \times \cU \to \CC^d$, such that
\begin{equation}\label{Eq:affine_property}
\int_D e^{\scal{\xi}{u}}p_t(x,d\xi) = \Phi(t,u) \exp(\scal{x}{\psi(t,u)})
\end{equation}
for all $t \in \Rplus$, $x \in D$ and $u \in \cU$.
\end{enumerate}
\end{defn}
\begin{rem}
Note that this definition does not specify $\psi(t,u)$ in a unique way. However there is a natural unique choice for $\psi$ that will be discussed in Prop.~\ref{Prop:Phipsi_properties} below. Also note that as long as $\Phi(t,u)$ is non-zero, there exists $\phi(t,u)$ such that $\Phi(t,u) = e^{\phi(t,u)}$ and \eqref{Eq:affine_property} becomes
\begin{equation}\label{Eq:affine_property_small_phi}
\int_D e^{\scal{\xi}{u}}p_t(x,d\xi) = \exp(\phi(t,u) +
\scal{x}{\psi(t,u)}).
\end{equation}
This is the essentially the definition that was used in \citet{Duffie2003}; with this notation the Fourier-Laplace transform is the exponential of an \emph{affine function} of $x$. This is usually interpreted as the reason for the name `affine process', even though affine functions also appear in other aspects of affine processes, e.g. in the coefficients of the infinitesimal generator, or in the differentiated semi-martingale characteristics. We use \eqref{Eq:affine_property} instead of \eqref{Eq:affine_property_small_phi}, as it leads to a slightly more general definition that avoids the a-priori assumption that the left hand side of \eqref{Eq:affine_property} is non-zero. Interestingly, in the paper of \citet{Kawazu1971} also the `big-$\Phi$' notation is used to define a `continuous-state branching process with immigration', which corresponds to an affine process on $\Rplus$ in our terminology.
\end{rem}
\begin{rem}It has recently been shown by \citet{Cuchiero2011} (see also \citet{Cuchiero2011a}), that any affine process on a general state space $D$ has a c\`adl\`ag modification under every $\PP^x, x \in D$. Moreover, when $X$ is an affine process relative to an arbitrary filtration $\FF_0$, then the $\PP^x$-augmentation $\FF^x$ of $\FF_0$ is right-continuous, for any $x \in D$. This implies that the assumptions that we make on the path properties of $X$ are in fact automatically satisfied after a suitable modification of the process.
\end{rem}

Before we explore the first consequences of
Definition~\ref{Def:affine_process}, we introduce some additional notation. For any $u \in \cU$ define
\begin{align}\label{Eq:sigma}
\sigma(u) &:= \inf \set{t \ge 0: \Phi(t,u) = 0},\\
\intertext{and}
\label{Eq:Q}
\cQ_k &:= \set{(t,u) \in \Rplus \times \cU_k: t < \sigma(u)},
\end{align}
for $ k \in\NN $. We set $ \cQ := \cup_k \cQ_k $. Finally on $\cQ$ let $\phi$ be a function such that
\[\Phi(t,u) = e^{\phi(t,u)}\qquad \quad ((t,u) \in \cQ).\] 
The uniqueness of $\phi$ will be discussed. The functions $\phi$ and $\psi$ have the following properties:

\begin{prop}\label{Prop:Phipsi_properties}
Let $X$ be an affine process on $D$. Then
\begin{enumerate}[(i)]
\item \label{Item:sigma_pos} It holds that $\sigma(u) > 0$ for any $u \in \cU$.
\item The functions $\phi$ and $\psi$ are uniquely defined on $\cQ$ by requiring that they are jointly continuous on $ \cQ_k $ for $ k \in \NN $ and satisfy $\phi(0,0) = \psi(0,0) = 0$.
\label{Item:uniqueness}
\item The function $\psi$ maps $\cQ$ into $\cU$. \label{Item:phipsi_range}
\item The functions $\phi$ and $\psi$ satisfy the \label{Item:semiflow}\emph{semi-flow} property. For any $u \in \cU$ and $t,s \ge 0$ with $(t + s,u) \in \cQ$  and $(s,\psi(t,u)) \in \cQ$ it holds that
\begin{equation}\label{Eq:flow_prop}
\begin{split}
\phi(t+s,u) &= \phi(t,u) + \phi(s,\psi(t,u)), \quad \phi(0,u) = 0\\
\psi(t+s,u) &= \psi(s,\psi(t,u)), \phantom{+ \phi(t,u)}\quad \psi(0,u) = u\,
\end{split}
\end{equation}
\end{enumerate}
\end{prop}
\begin{proof}
Choose some $x \in D$, and for $(t,u) \in \Rplus \times \cU$ define the function
\begin{equation}\label{Eq:affine_prop_repeat}
f(t,u) = \Phi(t,u)e^{\scal{\psi(t,u)}{x}} = \int_D e^{\scal{u}{\xi}}p_t(x,d\xi).
\end{equation}
Fix $k \in \NN$ and let $(t_n,u_n)_{n \in \NN}$ be a sequence in $\Rplus \times \cU_k$ converging to $(t,u) \in \Rplus \times \cU_k$. For any $\epsilon > 0$ we can find a function $\rho: D \to [0,1]$ with compact support, such that $\int_D (1 - \rho(\xi))p_t(x,d\xi) < \epsilon$. Moreover, there exists a $N_0 \in \NN$ such that
\[\left|e^{\scal{u_n}{\xi}} - e^{\scal{u}{\xi}}\right| < \epsilon, \quad \forall\, n \ge N_0, \xi \in \supp \rho.\]
By stochastic continuity of $p_t(x,d\xi)$ we can find $N_1 \ge N_0$ such that
\[\int_D (1 - \rho(\xi)) p_{t_n}(x,d\xi) < \epsilon, \quad \forall\,n \ge N_1,\]
and also
\[\left|\int_D e^{\scal{u}{\xi}} p_{t_n}(x,d\xi) - \int_D e^{\scal{u}{\xi}} p_t(x,d\xi) \right| < \epsilon, \quad \forall\,n \ge N_1,\]
For $n \ge N_1$, we now have
\begin{align*}
\left|f(t_n,u_n) - f(t,u) \right| &= \left|\int_D e^{\scal{u_n}{\xi}}p_{t_n}(x,d\xi) - \int_D e^{\scal{u}{\xi}}p_{t}(x,d\xi) \right| \le \\
&\le \left|\int_D e^{\scal{u_n}{\xi}} \rho(\xi) p_{t_n}(x,d\xi) - \int_D e^{\scal{u}{\xi}} \rho(\xi)p_{t_n}(x,d\xi) \right| + \\ &+ \left|\int_D e^{\scal{u_n}{\xi}}(1 - \rho(\xi)) p_{t_n}(x,d\xi) - \int_D e^{\scal{u}{\xi}}(1 - \rho(\xi))p_{t_n}(x,d\xi) \right| +  \\
&+ \left|\int_D e^{\scal{u}{\xi}} p_{t_n}(x,d\xi) - \int_D e^{\scal{u}{\xi}}  p_t(x,d\xi) \right| \le \\ &\le 
\epsilon + k \epsilon + \epsilon = \epsilon (2 + k).
\end{align*}
Since $\epsilon$ was arbitrary this shows the continuity of $f(t,u)$ on $\Rplus \times \cU_k$. Hence we conclude that $(t,u) \mapsto f(t,u)$ is continuous on $\Rplus \times \cU_k$ for each $ k \in \NN$. Moreover $f(t,u) = 0$ if and only if $\Phi(t,u) = 0$ and $f(0,u) = e^{\scal{u}{x}} \neq 0$ for all $u \in \cU$. We conclude from the continuity of $f$ that $\sigma(u) = \inf \set{t \ge 0: f(t,u) = 0} > 0$ for all $u \in \cU$ and \eqref{Item:sigma_pos} follows.

To obtain \eqref{Item:uniqueness}, note that for each $x \in D$, we have just shown that the
function $(t,u) \mapsto \int_D e^{\scal{u}{\xi}}p_t(x,\xi)$ maps
$\cQ_k$ \emph{continuously} into $\CC \setminus \set{0}$ for each $ k \in \NN$. We claim that the
mapping has a unique continuous logarithm\footnote{We adapt a proof from
\citet[Thm.2.5]{Bucchianico1991} to our setting.}, i.e.~for each $x \in D$ there exists
a unique function $g(x;,.,): \cQ \to \CC$ being continuous on $ \cQ_k $ for $ k \in \NN $, such that $g(x;0,0) = 0$ and
$\int_D e^{\scal{u}{\xi}}p_t(x,\xi) = e^{g(x;t,u)}$. For each $n \in \NN$ define the set
\[K_n = \set{(t,u): u \in \cU_n, \norm{u} \le n, t \in [0,\sigma(u) - 1/n]}.\]
Clearly, the $K_n$ are compact subsets of $\cQ_n \subset \cQ$ and exhaust $\cQ$ as $n \to \infty$. 
We show that every $K_n$ is contractible to $0$. Let $\gamma = (t(r),u(r))_{r \in [0,1]}$ be a
continuous curve in $K_n$. For each $\alpha \in [0,1]$ define
$\gamma_\alpha = (\alpha t(r),u(r))_{r \in [0,1]}$. Then
$\gamma_\alpha$ depends continuously on $\alpha$, stays in $K_n$ for
each $\alpha$ and satisfies $\gamma_1 = \gamma$ and $\gamma_0 =
(0,u(r))_{r \in [0,1]}$. Thus any continuous curve in $K_n$ is
homotopically equivalent to a continuous curve in $\set{0} \times
\cU$. Moreover, all continuous curves in $\set{0} \times \cU$ are
contractible to $0$, since $\cU$ is a convex cone. We conclude that
each $K_n$ is contractible to $0$ and in particular connected. Let $H_n: [0,1] \times K_n \to K_n$ be a
corresponding contraction, and for some fixed $x \in D$ write $f_n(t,u)$ for the restriction
of $(t,u) \mapsto \int_D e^{\scal{u}{\xi}}p_t(x,\xi)$ to $K_n$. Since
$H_n$ and $f_n$ are continuous and $K_n$ is compact, we have that
$\lim_{t \to s}\norm{f_n(H_n(t,.))  - f_n(H_n(s,.)}_\infty = 0$. Hence
$f_n \circ H_n$ is a continuous curve in $C_b(K_n)$ from $f_n$ to the
constant function $1$. By \citet[Thm.~1.3]{Bucchianico1991} there
exists a continuous logarithm $g_n \in C_b(K_n)$ that satisfies
$f_n(t,u) = e^{g_n(t,u)}$ for all $(t,u) \in K_n$. It follows that for arbitrary $m \le n$ in $\NN$ we have
\[g_m(t,u) = g_n(t,u) + 2\pi i \,l(t,u) \qquad \text{for all} \quad (t,u) \in K_m, \]
where $l(t,u)$ is a continuous function from $K_m$ to $\mathbb{Z}$ satisfying $l(0,0) = 0$. But $K_m$ is connected, hence also the image of $K_m$ under $l$. We conclude that $l(t,u) = 0$, and that $g_m(t,u) = g_n(t,u)$ for all $(t,u) \in K_m$. Taking $m = n$ this shows that $g_n$ is uniquely defined on each subset $K_n$ of $\cQ$. Taking $m < n$ it shows that $g_n$ extends $g_m$. Since the $(K_n)_{n \in \NN}$ exhaust $\cQ$, it follows that there exists indeed, for each $x \in D$, a unique function $g(x;.): \cQ \to \CC$ such that $g(x;0,0) = 0$ and  $\int_D e^{\scal{u}{\xi}}p_t(x,\xi) = e^{g(x;t,u)}$. Due to \eqref{Eq:affine_property} $g(x;t,u)$ must be of the form $\phi(t,u) + \scal{\psi(t,u)}{x}$, and since $D$ affinely spans $\RR^d$ also $\phi(t,u)$ and $\psi(t,u)$ are jointly continuous on $ \cQ_k $ for $ k \in \NN $ and uniquely determined on $\cQ$, whence we have shown \eqref{Item:uniqueness}.

Next note that the rightmost term of \eqref{Eq:affine_prop_repeat}
is uniformly bounded for all $x \in D$. Thus also the middle term is,
and we obtain that $\psi(t,u) \in \cU$, as claimed in
\eqref{Item:phipsi_range}. Applying the Chapman-Kolmogorov equation
to \eqref{Eq:affine_property} and writing $\Phi(t,u) =
e^{\phi(t,u)}$ yields that
\begin{multline}\label{Eq:Chapman_Phi}
\exp\left(\phi(t+s,u) + \scal{x}{\psi(t+s,u)}\right) = \int_D e^{\scal{\xi}{u}}p_{t+s}(x,d\xi) = \\
= \int_D p_s(x,dy) \int_D e^{\scal{\xi}{u}}p_t(y,d\xi) = e^{\phi(t,u)} \int_D e^{\scal{y}{\psi(t,u)}} p_s(x,dy) = \\
= \exp\left(\phi(t,u)+ \phi(s,\psi(t,u)) + \scal{x}{\psi(s,\psi(t,u)))}\right)
\end{multline}
for all $x \in D$ and for all $u \in \cU$ such that $(t + s,u) \in \cQ$ and $(s,\psi(t,u)) \in \cQ$ . Taking (continuous) logarithms on both sides  \eqref{Item:semiflow} follows.
\end{proof}
\begin{rem}
From now on $\phi$ and $\psi$ shall always refer to the unique choice of functions described in
Proposition~\ref{Prop:Phipsi_properties}.
\end{rem}

\section{Main Results}
\label{Sec:regularity}
\subsection{Definition and consequences of regularity}

We now introduce the important notion of \emph{regularity}.
\begin{defn}\label{Def:regular}
An affine process $X$ is called \emph{regular} if the derivatives
\begin{align}\label{Eq:FR_def}
F(u) = \frac{\partial \phi(t,u)}{\partial t}\Bigg |_{t=0+}, \qquad
R(u) = \frac{\partial \psi(t,u)}{\partial t}\Bigg |_{t=0+}
\end{align}
exist for all $u \in \cU$ and are continuous on $ \cU_k $ for each $ k \in \NN $.
\end{defn}
\begin{rem}
Note that in comparison with the definition given in the introduction, we now define $F(u)$ as the derivative at $t = 0$ of $t \mapsto \phi(t,u)$ instead of $t \mapsto \Phi(t,u)$. In light of Proposition~\ref{Prop:Phipsi_properties} these definitions coincide, since $\phi(t,u)$ is always defined for $t$ small enough and satisfies $\Phi(t,u) = e^{\phi(t,u)}$ with $\phi(0,u) = 0$.
\end{rem}

The next result illustrates why regularity is a crucial property; it has originally been established by \citet{Duffie2003} for affine processes on the state-space $\RR^n \times \Rplus^m$.
\begin{prop}\label{Prop:LevyK}
 Let $X$ be a \emph{regular} affine process. Then there exist $\RR^d$-vectors $b,\beta^1,\ldots,\beta^d$; $d \times d$-matrices $a,\alpha^1,\ldots,\alpha^d$; real numbers $c,\gamma^1,\ldots,\gamma^d$ and signed Borel measures $m,\mu^1,\ldots,\mu^d$ on $\RR^d \setminus \set{0}$, such that for all $u \in \cU$ the functions $F(u)$ and $R(u)$ can be written as
\begin{subequations}\label{Eq:FR_LK_form}
\begin{align}
F(u) &= \frac{1}{2}\scal{u}{a u} + \scal{b}{u} - c + \int_{\RR^d
\setminus \set{0}}{\left(e^{\scal{\xi}{u}} - 1 - \scal{
h(\xi)}{u}\right)\,m(d\xi)}\;,\\
R_i(u) &= \frac{1}{2}\scal{u}{\alpha^i u} + \scal{\beta^i}{u} -
\gamma^i + \int_{\RR^d \setminus \set{0}}{\left(e^{\scal{\xi}{u}} -
1 - \scal{ h(\xi)}{u}\right)\,\mu^i(d\xi)}\;,
\end{align}
\end{subequations}
with truncation function $h(x) = x \Ind{\norm{x} \le 1}$, and such that for all $x \in D$ the quantities 
\begin{subequations}\label{Eq:AB_nu}
\begin{align}
A(x) &= a + x_1 \alpha^1 + \dotsm + x_d \alpha^d,\\
B(x) &= b + x_1 \beta^1 + \dotsm + x_d \beta^d,\\
C(x) &= c + x_1 \gamma^1 + \dotsm + x_d \gamma^d,\\
\nu(x,d\xi) &= m(d\xi) + x_1 \mu^1(d\xi) + \dotsm + x_d \mu^d(d\xi)
\end{align}
\end{subequations}
have the following properties: $A(x)$ is positive semidefinite, $C(x) \le 0$ and \linebreak $\int_{\RR^d \setminus \set{0}}{\left( \norm{\xi}^2 \wedge 1\right)} \nu(x,d\xi) < \infty$.\\
Moreover, for $u \in \cU$ the functions $\phi$ and $\psi$ satisfy
the ordinary differential equations
    \begin{subequations}\label{Eq:Riccati}
    \begin{align}
    \pd{}{t} \phi(t,u)&=F(\psi(t,u)), &\quad &\phi(0,u)=0\label{Eq:Riccati_F}\\
  \pd{}{t} \psi(t,u)&=R(\psi(t,u)), &\quad &\psi(0,u)=u\label{Eq:Riccati_R}.
    \end{align}
    \end{subequations}
    for all $t \in [0,\sigma(u))$.
\end{prop}
\begin{rem}
 The differential equations \eqref{Eq:Riccati} are called \emph{generalized Riccati equations}, since they are classical Riccati differential equations, when $m(d\xi)$ = $\mu^i(d\xi) = 0$. Moreover equations \eqref{Eq:FR_LK_form} and \eqref{Eq:AB_nu} imply that $u \mapsto F(u) + \scal{R(u)}{x}$ is a function of L\'evy-Khintchine form for each $x \in D$.
\end{rem}
\begin{proof}
The equations \eqref{Eq:Riccati} follow immediately by
differentiating the semi-flow equations \eqref{Eq:flow_prop}. The
form of $F,R$ follows by the following
argument: By \eqref{Eq:FR_def} and the affine property \eqref{Eq:affine_property} it holds for all $x \in D$ and $u \in \cU$ that
\begin{align}\label{Eq:A_general_form}
F(u) + \scal{x}{R(u)} &= \lim_{t \downarrow 0}\frac{1}{t}\left\{e^{\phi(t,u) + \scal{x}{\psi(t,u) - u}} - 1\right\} = \notag \\
&= \lim_{t \to 0}\frac{1}{t}\left\{\int_D {e^{\scal{\xi - x}{u}}p_t(x,d\xi)} - 1\right\} = \notag \\
 &= \lim_{t \to 0}\left\{\frac{1}{t}\int_D {\left(e^{\scal{\xi - x}{u}} - 1\right)\, p_t(x,d\xi)} + \frac{p_t(x,D) - 1}{t}\right\} = \notag \\
 &= \lim_{t \to 0}\left\{\frac{1}{t}\int_{D - x}{\left(e^{\scal{\xi}{u}} -
 1\right)\,\wt{p}_t(x,d\xi)}\right\} + \lim_{t \to 0} \frac{p_t(x,D) - 1}{t},
\end{align}
where we write $\wt{p}_t(x,d\xi) := p_t(x,d\xi + x)$ for the
shifted transition kernel. Inserting $u = 0$ into
the above equation shows that $\lim_{t \downarrow 0} (p_t(x,D) - 1)/t$ converges
to $F(0) + \scal{x}{R(0)}$. Set $c = - F(0)$ and $\gamma =
- R(0)$ and write $\wt{F}(u) = F(u) + c$ and $\wt{R}(u) = R(u) + \gamma$,
such that
\begin{equation}\label{Eq:Poisson_limit}
\exp\left(\wt{F}(u) + \scal{x}{\wt{R}(u)}\right) = \lim_{t \downarrow 0}\exp\left\{\frac{1}{t}\int_{D - x}{\left(e^{\scal{\xi}{u}} -
 1\right)\,\wt{p}_t(x,d\xi)}\right\}.
\end{equation}
For each $t \ge 0$ and $x \in D$, the exponential on the right hand side is the Fourier-Laplace transform of a compound Poisson distribution with jump measure $\wt{p}_t(x,d\xi)$ and jump intensity $\frac{1}{t}$ (cf.~\citet[Ch.~4]{Sato1999}). The Fourier-Laplace transforms converge pointwise for $u \in \cU$ -- and in particular for all $u \in i\RR^d$ -- as $t \to 0$. By the assumption of regularity the pointwise limit is continuous at $u = 0$ as function on $i\RR^d \subset \cU_k $ for each $ k \in \NN $, which implies by L\'evy's continuity theorem that the compound Poisson distributions converge weakly to a limiting probability distribution. Moreover, as the weak limit of compound Poisson distributions, the limiting distribution must be infinitely divisible. Let us denote the law of the limiting distribution, for given $x \in D$, by $K(x,dy)$. Since it is infinitely divisible, its characteristic exponent is of L\'evy-Khintchine form, and we obtain the identity
\begin{multline}\label{Eq:LK_decomp_interm}
 \wt{F}(u) + \scal{x}{\wt{R}(u)} = \log \int_{\RR^d}{e^{\scal{\xi}{u}}K(x,d\xi)} = \\
 = - \frac{1}{2}\scal{uA(x)}{u} + \scal{B(x)}{u} - \int_{\RR^d}{\left(e^{\scal{\xi}{u}} - 1 - \scal{h(\xi)}{u}\right)\nu(x,d\xi)},
\end{multline}
where for each $x \in D$, $A(x)$ is a positive semi-definite $d \times d$-matrix, $B(x) \in \RR^d$, and $\nu(x,d\xi)$ a $\sigma$-finite Borel measure on $\RR^d \setminus \set{0}$ and $\int \left(\norm{\xi}^2 \wedge 1\right) \nu(x,d\xi) < \infty$. Note that in the step from \eqref{Eq:Poisson_limit} from \eqref{Eq:LK_decomp_interm} we have used that $\wt{F}(u)$ and $\wt{R}(u)$ are continuous on every $\cU_k, k \in \NN$, and hence that $\wt{F}(u) + \scal{x}{\wt{R}(u)}$ is the unique continuous logarithm of $\exp(\wt{F}(u) + \scal{x}{\wt{R}(u)})$ on each $\cU_k$ and for all $x \in D$. Since \eqref{Eq:LK_decomp_interm} holds for all $x \in D$, and $D$ contains at least $d+1$ affinely independent points, we conclude that $A(x)$, $B(x)$ and $\nu(x,d\xi)$ are of the form given in \eqref{Eq:AB_nu} and the decompositions in \eqref{Eq:FR_LK_form} follow.
\end{proof}

In general, the parameters $(a,\alpha^i, b, \beta^i, c, \gamma^i, m,
\mu^i)_{i \in \set{1, \dotsc, d}}$ of $F$ and $R$ have to satisfy
additional conditions, called \emph{admissibility} conditions, that
guarantee the existence of an affine Markov process $X$ with state space $D$ and prescribed $F$ and $R$. It is clear that such conditions depend strongly
on the geometry of the (boundary of the) state space $D$. Finding such (necessary and sufficient) conditions on the parameters for different types of
state spaces has been the focus of several publications. For $D =
\Rplus^m \times \RR^n$ the admissibility conditions have been
derived by \citet{Duffie2003}, for $D = S_d^+$, the cone of
semi-definite matrices by \citet{Cuchiero2009}, and for cones $D$
that are symmetric and irreducible in the sense of \citet{Faraut1994} by \citet{CKMT2011}. Finally for affine diffusions ($m = \mu^i = 0$) on polyhedral cones and on quadratic state spaces the admissiblility conditions have been given by \citet{Spreij2010}. The purpose of this article is to show that there are parameters, in terms of which one can ask for admissibility conditions, but we do not aim to derive these admissibility conditions for conrete specifications of the state space $D$.

\subsection{Auxiliary Results}
For the sake of simpler notation we define
\[\varrho(t,u) = \psi(t,u) - u.\]
Note that we have $\varrho(0,u) = 0$ for all $u \in \cU$. The
following Lemma is a purely analytical result that will be needed
later.

\begin{lem}\label{Lem:sequence}
Let $K$ be a compact subset of $\cU_l$ for some $l \in \NN$ and assume that
\begin{equation}\label{Eq:bad_assumption}
\limsup_{t\to 0} \sup_{u \in K} \left(\frac{\abs{\phi(t,u)}}{t} + \frac{\norm{\varrho
(t,u)}}{t}\right) =\infty \, .
\end{equation}
Then there is $x\in D$, $\varepsilon >0$, $\eta >0$, $ z \in \mathbb{C}$ with $\abs{z}=1 $, a sequence
$(t_k)^\infty_{k=1}$ of positive real numbers, a sequence $(M_k)^\infty_{k=1}$ of integers satisfying
\begin{equation}\label{Eq:sequence_properties}
\lim_{k\to\infty} t_k =0, \quad\quad \lim_{k\to\infty} M_k=\infty, \quad\quad \lim_{k\to\infty} M_k t_k=0,
\end{equation}
and a sequence of complex vectors $(u_k)_{k=0}^\infty$ in $K$ such that $u_k \to u_0$ and
\begin{equation}\label{Eq:inequality_phipsi}
|\phi (t_k ,u_k)+\scal{x}{\varrho(t_k,u_k)} | \geq \eta \norm{\varrho(t_k,u_k)}.
\end{equation}
Moreover, for all $\xi\in \RR^d$ satisfying $\norm{x-\xi} <\varepsilon,$
\begin{equation}\label{Eq:equality_Mk}
M_k (\phi(t_k,u_k)+\scal{\xi}{\varrho(t_k,u_k)})=z+e_{k,\xi},
\end{equation}
where the complex numbers $e_{k,\xi}$ describing the deviation from $z$ satisfy $\abs{e_{k,\xi}}<\frac{1}{2}$ and $\lim_{k \to \infty} \sup_{\set{\xi: \norm{x - \xi} < \varepsilon}} |e_{k,\xi}| = 0$.
\end{lem}
\begin{rem}
 The essence of the above Lemma is that the behavior of $\phi(t,u)$ and $\varrho(t,u)$ as $t$ approaches $0$ can be crystallized along  the sequences $t_k$ and $M_k$. Equation \eqref{Eq:sequence_properties} then states that $t_k = o\left(\frac{1}{M_k}\right)$, and  \eqref{Eq:equality_Mk} asserts that the asymptotic equivalence
 \[\left|\phi(t_k,u_k)+\scal{\xi}{\varrho(t_k,u_k)}\right| \sim \frac{1}{M_k},\]
 holds uniformly for all $\xi$ in an $\varepsilon$-ball around $x$.
\end{rem}

\begin{proof} We first show all assertions of the Lemma for a
sequence $(\wt{M}_k)_{k \in \NN}$ of positive but not necessarily
integer numbers. In the last step of the proof we show that it is possible to switch from
$(\wt{M}_k)_{k \in \NN}$ to the integer sequence $(M_k)_{k \in
\NN}$.\\
By Assumption~\eqref{Eq:bad_assumption} we can find a sequence
$(t_k)_{k=0}^\infty \downarrow 0$ and a sequence
$(u_k)_{k=0}^\infty$ with $u_k \in K$, such that
\[\frac{|\phi(t_k,u_k)| + \norm{\varrho(t_k,u_k)}}{t_k} \to \infty\;.\]
Passing to a subsequence, and using the compactness of $K$, we may assume that $u_k$ converges to some point $u_0 \in K$.
For more concise notation, we write from now on $\phi_k =
\phi(t_k,u_k)$ and $\varrho_k = \varrho(t_k,u_k)$. Note that $\phi_k
\to 0$ and $\varrho_k \to 0$, by joint continuity of $\phi$ and $\varrho$ on $\cU_l$,
and the
fact that $\phi(0,u) = 0$ and $\varrho(0,u) = 0$.\\
Let us now show \eqref{Eq:inequality_phipsi}. By assumption, $D$ contains $d+1$ affinely
independent vectors $x_0, x_1, \dotsc, x_d$. Assume for a contradiction
that
\begin{equation}\label{Eq:limit0}
 \lim_{k \to \infty} \frac{\left|\phi_k + \scal{\varrho_k}{x_j}\right|}{\norm{\varrho_k}} \to 0
\end{equation}
for all $x_j$, $j
\in \set{0, \dotsc, d}$. Since the vectors $x_j$ affinely span
$\RR^d$, the vectors $\set{x_1 - x_0, \dots, x_d - x_0}$ are linearly independent, and we can find some numbers $\alpha_{j,k} \in \CC$, such that
\begin{equation}
\varrho_k/\norm{\varrho_k} = \sum_{j = 1}^d{\alpha_{j,k}}\left(x_j -
x_0\right),
\end{equation}
for all $k \in \NN$. Moreover, since $\varrho_k/\norm{\varrho_k}$ is
bounded also the $|\alpha_{j,k}|$ are bounded by a constant. By direct
calculation we obtain
\begin{equation}\label{Eq:sum}
\sum_{j=1}^d \alpha_{j,k} \left(\frac{\phi_k +
\scal{\varrho_k}{x_j}}{\norm{\varrho_k}} - \frac{\phi_k +
\scal{\varrho_k}{x_0}}{\norm{\varrho_k}}\right) =
\frac{\scal{\varrho_k/\norm{\varrho_k}}{\varrho_k}}{\norm{\varrho_k}}
= 1,
\end{equation}
for all $k \in \NN$. On the other hand, \eqref{Eq:limit0} implies that the left hand side of \eqref{Eq:sum} converges to $0$ as $k \to \infty$, which is a contradiction. We conclude that there exists $x^* \in D$ for which
\begin{equation}
 \frac{\left|\phi_k + \scal{\varrho_k}{x^*}\right|}{\norm{\varrho_k}} \geq \eta
\end{equation}
for some $\eta > 0$ after possibly passing to subsequences, whence \eqref{Eq:inequality_phipsi} follows.\\

To show \eqref{Eq:equality_Mk}, set $\wt{M}_k = \left|\phi_k +
\scal{x^*}{\varrho_k}\right|^{-1}$. Passing once more to a
subsequence, and using the compactness of the complex unit circle,
we can find some $\alpha \in [0,2\pi)$ such that $\arg \left(\phi_k
+ \scal{x^*}{\varrho_k}\right) \to \alpha$. Now
\[\phi_k + \scal{\xi}{\varrho_k} = (\phi_k  + \scal{x^*}{\varrho_k}) + \scal{\xi - x^*}{\varrho_k} = \frac{1}{\wt{M}_k}(e^{i\alpha} + e^{(1)}_k) + \scal{\xi - x^*}{\varrho_k}\]
where $e^{(1)}_k \to 0$ as $k \to \infty$. Multiplying by $\wt{M}_k$
and setting $z = e^{i \alpha}$ we obtain
\[\wt{M}_k (\phi_k + \scal{\xi}{\varrho_k}) = z + e^{(1)}_k + e^{(2)}_{k,\xi}\]
where we can estimate $|e^{(2)}_{k,\xi}| \le \wt{M}_k \varepsilon
\norm{\varrho_k}$. Since $\wt{M}_k \norm{\varrho_k} \le
\frac{1}{\eta}$ by \eqref{Eq:inequality_phipsi} we can make
$e^{(2)}_{k,\xi}$ arbitrarily small by choosing a small enough
$\varepsilon$. Setting $e_{k,\xi} = e_k^{(1)} + e^{(2)}_{k,\xi}$ we
obtain \eqref{Eq:equality_Mk}. Finally, for each $k \in \NN$ let
$M_k$ be the nearest integer greater than $\wt{M}_k$. It is clear
that after possibly removing a finite number of terms from all
sequences, the assertion of the Lemma is not affected from switching
from $\wt{M}_k$ to $M_k$.
\end{proof}

\begin{lem}\label{Lem:bigL}
Let $X={(X_t)}_{t \geq 0}$ be an affine process starting at $ X_0 $ and let $u \in \cU$, $\Delta
> 0$. Define
\begin{equation}\label{Eq:bigL}
L(n,\Delta,u) = \exp \left( \scal{u}{X_{n \Delta} - X_0} -
\sum_{j=1}^n\left(\phi(\Delta, u) + \scal{\varrho(\Delta,
u)}{X_{(j-1) \Delta}}\right)\right).
\end{equation}
Then $n \mapsto L(n,\Delta, u)$ is a $(\cF_{n \Delta})_{n \in
\NN}$-martingale under every measure $\PP^x$, $ x \in D$.
\end{lem}
\begin{proof}
It is obvious that each $L(n,\Delta,u)$ is $\cF_{n \Delta}$-measurable. We show the martingale property by combining the affine property of $X$ with the tower law for conditional expectations. Write
\[S_n = \sum_{j=1}^n\left(\phi(\Delta, u) + \scal{\varrho(\Delta, u)}{X_{(j-1) \Delta}}\right),\]
and note that $S_n$ is $\cF_{(n-1) \Delta}$-measurable. We have that
\begin{align*}
\Excond{x}{L(n,\Delta,u)}{\cF_{(n-1) \Delta}} &= \Excond{x}{\exp\left(\scal{u}{X_{n \Delta} - X_0}\right)}{\cF_{(n-1) \Delta}}e^{-S_n} = \\ &= \exp\left(\phi(\Delta,u) + \scal{\psi(\Delta,u)}{X_{(n-1) \Delta}} - \scal{u}{X_0} - S_n\right) = \\ &= \exp\left(\scal{u}{X_{(n-1) \Delta} - X_0} - S_{n-1}\right) = L(n-1,\Delta,u),
\end{align*}
showing that $n \mapsto L(n,\Delta,u)$ is indeed a $(\cF_{n
\Delta})_{n \in \NN}$-martingale under every $\PP^x, x \in D$.
\end{proof}

We combine the two preceding Lemmas to show the following.
\begin{prop}\label{Prop:finite}
Let $X$ be a c\`adl\`ag affine process. Then the associated functions
$\phi(t,u)$ and $\varrho(t,u) = \psi(t,u) - u$ satisfy
 \begin{equation}\label{Eq:supsup_finite}
  \limsup_{t \downarrow 0} \sup_{u \in K} \left(\frac{\abs{\phi(t,u)}}{t} + \frac{\norm{\varrho
(t,u)}}{t}\right) < \infty\, 
 \end{equation}
for each compact subset $K$ of $\cU_l$ and each $ l \in \NN $.
\end{prop}
\begin{proof}
 We argue by contradiction: Fix $ l \in \NN $ and assume that \eqref{Eq:supsup_finite} fails to hold true. Then by Lemma~\ref{Lem:sequence} there exist $\varepsilon > 0$ and sequences $u_k \to u_0$ in $ K $, $t_k \downarrow 0$ and $M_k \uparrow \infty$ such that $t_k M_k \to 0$ and equations \eqref{Eq:inequality_phipsi}, \eqref{Eq:equality_Mk} hold.
 Define the $(\cF_{n
\Delta})_{n \in \NN}$-stopping times $N_k = \inf \set{n \in \NN: \norm{X_{n t_k} - X_0} > \varepsilon }$. Then by Lemma~\ref{Lem:bigL} and Doob's optional stopping lemma we know that
\begin{multline}\label{Eq:bigL2}
 n \mapsto L(n \wedge N_k,t_k,u_k) = \\ = \exp \left( \scal{u_k}{X_{(n \wedge N_k) t_k} - X_0} - \sum_{j=1}^{n \wedge N_k}\left(\phi(t_k, u_k) + \scal{\varrho(t_k, u_k)}{X_{(j-1) t_k}}\right)\right)
\end{multline}
is a $(\cF_{n
\Delta})_{n \in \NN}$-martingale too. It follows in particular that $\E{L(M_k \wedge
N_k,t_k,u_k)} = 1$ for all $k \in \NN$. By \eqref{Eq:equality_Mk},
we have the uniform bound
\begin{align}
|L(M_k \wedge N_k,t_k,u_k)| &\le C \exp\left(\left|\sum_{j=1}^{M_k
\wedge N_k}\left(\phi(t_k, u_k) + \scal{\varrho(t_k,
u_k)}{X_{(j-1)t_k}}\right)\right|\right) \le \notag \\
&\le C\exp(3/2) \label{Eq:bound_for_L},
\end{align}
where $C = \exp\left(-\Re \scal{u}{X_0}\right)$. Let $\delta > 0$
and $x \in D$. Since $X$ is c\`adl\`ag we can find a $T
> 0$ such that $\PP^x\left(\sup_{t \in [0,T]} \norm{X_t - X_0} >
\varepsilon \right) < \delta$. For $k$ large enough $t_k M_k \le T$
and hence $\PP(M_k > N_k) < \delta$. We conclude that
$\PP^x\left(\lim_{k \to \infty} \tfrac{M_K \wedge N_k}{M_k} =
1\right) \ge 1 - \delta$, and since $\delta$ was arbitrary $\lim_{k
\to \infty} \tfrac{M_k \wedge N_k}{M_k} = 1$ holds $\PP^x$-a.s. for
any $x \in D$. Together with \eqref{Eq:equality_Mk} and
\eqref{Eq:bound_for_L} we obtain by dominated convergence that
\begin{multline}
\lim_{k \to \infty}\Ex{x}{L(M_k \wedge N_k,T_k,u_k)} = \Ex{x}{\lim_{k \to \infty} L(M_k \wedge N_k,T_k,u_k)} = \\
= \Ex{x}{\lim_{k \to \infty} \exp\left((M_k \wedge N_k)
\left(\phi(t_k,u_k) + \scal{\varrho(t_k,u_k)}{x}\right)\right)} =
e^{-z}.
\end{multline}
where $|z| = 1$. But $\Ex{x}{L(M_k \wedge N_k, T_k,u_k)} = 1$ by its martingale property, which is the desired contradiction.
\end{proof}

\subsection{Affine processes are regular}
\begin{lem}\label{Lem:subsequence}
Let a sequence $t_k(u) \downarrow 0$ be assigned to each $u \in
\cU$. Then each of these sequences has a subsequence $\bbS(u) := \left(s_k(u)\right)_{k \in \NN}$
such that the limits
\begin{equation}\label{Eq:FR_limit_repeat}
F_\bbS(u) := \lim_{s_k(u) \downarrow 0}
\frac{\phi(s_k(u),u)}{s_k(u)}, \qquad R_\bbS(u) := \lim_{s_k(u)
\downarrow 0} \frac{\varrho(s_k(u),u)}{s_k(u)}
\end{equation}
are well-defined and finite. Moreover the subsequences $\bbS(u)$
can be chosen such that the numbers $F_\bbS(u)$ and $R_\bbS(u)$ are bounded on each compact subset $K$ of $\cU_l$ for each $ l \in \NN $.
\end{lem}
\begin{proof}
Let the sequences $t_k(u) \downarrow 0$ be given, but assume that the assertion of the Lemma
does not hold true. Then either $t_k(u)$ for some $ u \in \cU $ has no subsequence for which the limits in
\eqref{Eq:FR_limit_repeat} exist, or the limits $F(u)$ and $R(u)$ exist for
each $u \in \cU$, but at least one of them is not bounded in some compact $K \subset U_l$ for some $ l \in \NN $.\\
Consider the first case. By the Bolzano-Weierstrass theorem an
$\RR^d$-valued sequence that contains no convergent subsequence must
be unbounded, and we conclude that
\[ \limsup_{t_k(u) \downarrow 0}
\left(\frac{\abs{\phi(t_k(u),u)}}{t_k(u)} + \frac{\norm{\varrho
(t_k(u),u)}}{t_k(u)}\right) = \infty,\] in contradiction to
Proposition~\ref{Prop:finite}. Consider now the second assertion. Fix $ l \in \NN $. For each $u \in \cU_l$ there is a sequence $s_k(u)$ such that \eqref{Eq:FR_limit_repeat} holds, but $F_\bbS(u)$
or $R_\bbS(u)$ is not bounded in $K \subset \cU_l$, i.e. there exists a sequence
$u_n \to u_0$ in $ K $ for which $|F(u_n)| + \norm{R(u_n)} \to \infty$. Fix some $\eta > 0$. Then for each $k \in \NN$ there exists an $N_k \in \NN$ such that
\[\left|\frac{\phi(s_{N_k}(u_k),u_k)}{s_{N_k}(u_k)}\right| \ge \left|F(u_k)\right| - \eta/2 \qquad \text{and} \qquad \norm{\frac{\varrho(s_{N_k}(u_k),u_k)}{s_{N_k}(u_k)}} \ge \norm{R(u_k)} - \eta/2.\]
We conclude that
\begin{equation*}
\limsup_{s_k \downarrow 0} \sup_{u \in K}
\left(\frac{\abs{\phi(s_k,u)}}{s_k} + \frac{\norm{\varrho
(s_k,u)}}{s_k}\right) \ge \limsup_{k \to \infty} |F(u_k)| + \norm{R(u_k)} - \eta   = \infty,
\end{equation*}
again in contradiction to Prop.~\ref{Prop:finite}.
\end{proof}
Having shown Lemma~\ref{Lem:subsequence}, only a small step remains
to show regularity. Comparing with Definition \ref{Def:regular} we
see that two ingredients are missing: First we have to show that
the limits $F(u)$ and $R(u)$ do not depend on the choice of
subsequence, i.e. they are proper limits and hence the proper
derivatives of $\phi$ and $\psi$ at $t = 0$, and second we have to show that $F$
and $R$ are continuous on $ \cU_l$ for each $ l \in \NN$.

\begin{thm}\label{Thm:main}
Let $X$ be a c\`adl\`ag affine process on $D \subset \RR^d$. Then $X$ is
regular.
\end{thm}
\begin{proof}
Our first step is to show that the derivatives $F(u)$ and $R(u)$ in \eqref{Eq:FR_def} exist. By
Lemma~\ref{Lem:subsequence} we already know that they exist as limits along a
sequence $\bbS(u)$ which depends on the point $u \in \cU$ and has been chosen as a particular subsequence of a given sequence $(t_k(u))_{k \in \NN}$. We show
now that the limit is in fact independent of the choice of $\bbS(u)$ and even of the original sequence $(t_k(u))_{k \in \NN}$, and hence that $F(u)$ and $R(u)$ are proper derivatives in the sense of \eqref{Eq:FR_def}. To this end, fix some $u \in \cU$, and let $\wt{\bbS}(u)$
be an arbitrary other sequence $\wt{s}_k(u) \downarrow 0$, such that
\begin{equation}\label{Eq:FR_limit}
\wt{F}_{\bbS}(u) := \lim_{\wt{s}_k(u) \downarrow 0}
\frac{\phi(\wt{s}_k(u),u)}{\wt{s}_k(u)}, \qquad \wt{R}_{\bbS}(u) :=
\lim_{\wt{s}_k(u) \downarrow 0} \frac{\varrho(\wt{s}_k(u),u)}{\wt{s}_k(u)}.
\end{equation}
We want to show that $F_{\bbS}(u) = \wt{F}_{\bbS}(u)$ and $R_{\bbS}(u) =
\wt{R}_{\bbS}(u)$. Assume for a contradiction that this were not the
case. Then we can find $x \in D$ and $r > 0$ such that the convex
set $\set{F_{\bbS}(u) + \scal{R_{\bbS}(u)}{\xi}: \norm{\xi -x} \le
r}$ and its counterpart involving $\wt{\bbS}$ are disjoint, i.e.
\begin{multline}\label{Eq:empty_intersection}
\set{\vphantom{\sum}F_{\bbS}(u) + \scal{R_{\bbS}(u)}{\xi}: \norm{\xi
-x} \le r} \cap \set{\vphantom{\sum}\wt{F}_{\bbS}(u) +
\scal{\wt{R}_{\bbS}(u)}{\xi}: \norm{\xi -x} \le r} = \emptyset.
\end{multline}
For the next part of the proof, we set $\tau = \inf \set{t \ge 0:
\norm{X_t - X_0} \ge r}$,  and introduce the following notation:
\begin{eqnarray*}
a^u_t &:=& F_{\bbS}(u) + \scal{R_{\bbS}(u)}{X_{t}}, \qquad \qquad A^u_t := \int_0^{t}{a^u_{s-}ds},\\
G^u_t &:=& \exp(A^u_t), \qquad \qquad \qquad Y^u_t :=
\exp(\scal{u}{X_{t} - X_0}
\end{eqnarray*}
with $\wt{a}^u_t$, $\wt{A}^u_t$ and $\wt{G}^u_t$ the corresponding counterparts for $\wt{F}_{\bbS}$ and $\wt{R}_{\bbS}$.
We show that
\begin{equation}\label{Eq:Ldef}
L^u_{t \wedge \tau} = \frac{Y^u_{t \wedge \tau}}{G^u_{t \wedge
\tau}} = \exp \left( \scal{u}{X_{t \wedge \tau} - X_0} - \int_0^{t
\wedge \tau}\left(F_\bbS(u) +
\scal{R_\bbS(u)}{X_{s-}}\right)ds\right)
\end{equation}
is a martingale under every $\PP^x, x\in D$. This reduces to showing that
\[
\Ex{x}{\exp \left(\scal{u}{X_{h \wedge \tau} - X_0} - \int_{0}^{h \wedge \tau }\left(F_\bbS(u) + \scal{R_\bbS(u)}{X_{s-}}\right)ds\right)}=1 \, ,
\]
since then by the Markov property of $X$
\begin{multline}
\Excond{x}{\exp \left( \scal{u}{X_{(t+h) \wedge \tau} - X_{t \wedge \tau}} - \int_{t \wedge \tau}^{(t+h) \wedge \tau}\left(F_\bbS(u) + \scal{R_\bbS(u)}{X_{s-}}\right)ds\right)}{\cF_t} = \\
\Excond{x}{\exp \left( \scal{u}{X_{(t+h) \wedge \tau} - X_{t \wedge \tau}} - \int_{t \wedge \tau}^{(t+h) \wedge \tau}\left(F_\bbS(u) + \scal{R_\bbS(u)}{X_{s-}}\right)ds\right)1_{\tau \geq t}}{\cF_t}+1_{\tau \leq t} = \\
= \Ex{X_t}{\exp \left(\scal{u}{X_{h \wedge \tau} - X_0} - \int_{0}^{h \wedge \tau }\left(F_\bbS(u) + \scal{R_\bbS(u)}{X_{s-}}\right)ds\right)}1_{\tau \geq t}+1_{\tau \leq t} = 1
\end{multline}
holds true. Now, use the sequence $\bbS(u) = (s_n(u))_{n \in \NN} \downarrow 0$
to define a sequence of Riemannian sums approximating the above
integral. Define $M_k = \lfloor h/s_k \rfloor$ and $N_k = \inf
\set{n \in \NN: \norm{X_{n s_k} - X_0} > r}$. First we show that
$s_k N_k \to \tau$ almost surely under every $\PP^x$. Fix $\omega
\in \Omega$ such that $t \to X_t(\omega)$ is a c\`adl\`ag function. Let
$\wt{N}_k(\omega)$ be a sequence in $\NN$ such that $s_k
\wt{N}_k(\omega) \downarrow \tau(\omega)$. It follows from the
right-continuity of $t \mapsto X_t(\omega)$ that for large enough
$k$ it holds that $\norm{X_{s_k \wt{N}_k} - X_0} > r$ and hence that
eventually $\wt{N}_k(\omega) \ge N_k(\omega)$. On the other hand
$\norm{X_{s_k N_k} - X_0} > r$ for all $k \in \NN$, which implies
that $N_k(\omega) s_k \ge \tau(\omega)$. Hence, for large enough $k
\in \NN$ it holds that
\[s_k \wt{N}_k(\omega) \ge s_k N_k(\omega) \ge \tau(\omega).\]
We also know that $s_k \wt{N}_k(\omega) \to \tau(\omega)$ as $k \to
\infty$, such that we conclude that $s_k N_k \to \tau$
$\PP^x$-almost surely, as claimed. By Riemann approximation and the
fact that $X$ is c\`adl\`ag it then holds that
\[\sum_{j=1}^{M_k \wedge N_k} \left(F_\bbS(u) + \scal{R_\bbS(u)}{X_{(j-1) s_k}}\right) s_k \to \int_{0}^{h \wedge \tau}\left(F_\bbS(u) + \scal{R_\bbS(u)}{X_{s-}}\right)ds\]
$\PP^x$-almost-surely as $k \to \infty$ for all $x \in D$.

From Lemma~\ref{Lem:subsequence} we know that $\phi(s_k,u) = F_\bbS(u) s_k  + o(s_k)$ and $\phi(s_k,u) = R_\bbS(u) s_k  + o(s_k)$. Moreover $(M_k \wedge N_k) o(s_k) \to 0$ since $M_k s_k \to 0$. Thus we have that
\begin{multline*}
L(M_k \wedge N_k,s_k,u) = \\
= \exp\left(\scal{u}{X_{(M_k \wedge N_k)s_k} - X_0} - \sum_{j=1}^{M_k \wedge N_k} \left(\phi(t_k,u) + \scal{\varrho(t_k,u)}{X_{(j-1)s_k}}\right)\right) = \\
= \exp\left(\scal{u}{X_{(M_k \wedge N_k)s_k} - X_0} - \vphantom{\sum_{j=1}^{M_k \wedge N_k}}\right. \\
- \left.\sum_{j=1}^{M_k \wedge N_k} \left(F_\bbS(u) + \scal{R_\bbS(u)}{X_{(j-1)s_k}}\right) s_k + (M_k \wedge N_k) o(s_k)\right) \to \\
\to \exp \left(\scal{u}{X_{h \wedge \tau} - X_0} - \int_{0}^{h
\wedge \tau}\left(F_\bbS(u) +
\scal{R_\bbS(u)}{X_{s-}}\right)ds\right),
\end{multline*}
as $k \to \infty$ almost surely with respect to all $\PP^x, x\in D$.
But by Lemma~\ref{Lem:bigL} and optional stopping, $\Ex{x}{L(M_k
\wedge N_k,s_k,u)} = 1$, such that by dominated convergence we
conclude that
\[ \E{\exp \left(\scal{u}{X_{h \wedge \tau} - X_0} - \int_{0}^{h \wedge \tau}\left(F_\bbS(u) + \scal{R_\bbS(u)}{X_{s-}}\right)ds\right)} = 1,\]
and hence that $t \mapsto L^u_{t \wedge \tau}$ is a martingale.
Summing up we have established that $Y^u_{t \wedge \tau} = L^u_{t
\wedge \tau} G^u_{t \wedge \tau}$, where $L^u_{t \wedge \tau}$ is a
martingale and hence a semimartingale. Clearly, the process $G^u_{t
\wedge \tau}$ is predictable and of finite variation and hence a
semimartingale too. We conclude that also the product $Y^u_{t \wedge
\tau} = \exp\left(\scal{u}{X_{t \wedge \tau}^x - x}\right)$ is a
semimartingale. It follows from \citet[Thm.~I.4.49]{Jacod1987} that
$M^u_{t \wedge \tau} = Y^u_{t \wedge \tau} - \int_0^{t \wedge
\tau}{L^u_{s-}}dG^u_s$ is a local martingale. We can rewrite $M^u_t$
as
\begin{equation*}
M^u_t = Y^u_{t \wedge \tau} - \int_0^t{L^u_{s-}G^u_{s-}dA^u_s} =
Y^u_t - \int_0^t{Y^u_{s-}dA^u_s} = Y^u_t -
\int_0^t{Y^u_{s-}a^u_{s-}ds}.\end{equation*} Hence $Y^u_{t \wedge
\tau} = M^u_{t \wedge \tau} + \int_0^{t \wedge
\tau}{Y^u_{s-}a^u_{s-}ds}$ is the decomposition of the
semi-martingale $Y^u_{t \wedge \tau}$ into a local martingale and a
finite variation part. But $\int_0^{t \wedge
\tau}{Y^u_{s-}a^u_{s-}ds}$ is even predictable, such that $Y^u$ is a
special semi-martingale, and the decomposition is unique. The same
derivation goes through with $A^u$ replaced by $\wt{A}^u$ and by the
uniqueness of the special semi-martingale decomposition we conclude
that
\[\int_0^{t \wedge \tau}{Y^u_{s-}a^u_{s-}ds} = \int_0^{t \wedge \tau}{Y^u_{s-}\wt{a}^u_{s-}ds},\]
up to a $\PP^x$-nullset. Taking derivatives we see that
$Y^u_{t-}a^u_{t-} = Y^u_{t-}\wt{a}^u_{t-}$ on $\set{t \le \tau}$. As
long as $t \le \tau$ it holds that $Y^u_{t-} \neq 0$, and dividing
by $Y^u_{t-}$, we see that $a^u_{t-} = \wt{a}^u_{t-}$, that is
\[F_{\bbS}(u) + \scal{R_{\bbS}(u)}{X_{(t \wedge \tau)-}} = \wt{F}_{\bbS}(u) + \scal{\wt{R}_{\bbS}(u)}{X_{(t \wedge \tau)-}} \quad \text{for all} \quad t \le \tau,\]
$\PP^x$-a.s, in contradiction to \eqref{Eq:empty_intersection}. We
conclude that the limits $F_{\bbS}$ and $R_{\bbS}$ are independent
from the sequence $\bbS$, and hence that $F(u)$ and $R(u)$ exist as
proper derivatives in the sense of \eqref{Eq:FR_def}.

 It remains to show that $F(u)$ and $R(u)$ are continuous on $ \cU_l $ for each $ l \in \NN$. Fix $ l \in \NN $ and suppose for a contradiction that there exists a sequence $ u_k \to u_0 $ in $ \cU_l $ such that $ F(u_k) \to F^{\ast} $ and $ R(u_k) \to R^{\ast} $, such that either $F(u_0) \neq F^\ast$ or $R(u_0) \neq R^\ast$. Since $D$ affinely spans $\RR^d$ this means that there is $ x \in D $ with
\[F(u_0) + \scal{R(u_0)}{x} \neq F^{\ast} + \scal{R^{\ast}}{x}.\]
Using the fact that $\Ex{x}{L^{u_k}_{t \wedge \tau}} = 1$ for all $k
\in \NN$ we obtain
\begin{multline}
\frac{1}{t} \left(\exp(\scal{\phi(t,u_0) + \psi(t,u_0)}{x}) - 1\right) = \lim_{k \to
\infty} \frac{1}{t} \Ex{x}{e^{\scal{u_0}{X_t - X_0}} - L_{t \wedge
\tau}^{u_k}} = \\ = \lim_{k \to \infty} \frac{1}{t}\Ex{x}{e^{\scal{u_0}{X_t - X_0}} \left(1
- \exp(-\int_0^{t \wedge \tau} (F(u_k) + \scal{R(u_k)}{X_{s-}})
ds\right)} = \\ = \Ex{x}{\frac{1}{t} e^{\scal{u_0}{X_t - X_0}} \left(1
- \exp(-\int_0^{t \wedge \tau} (F^{\ast} + \scal{R^{\ast}}{X_{s-}})
ds\right)}.\label{Eq:Fstar}
\end{multline}
for all $t \le \sigma(0)$ by dominated convergence. Writing $C =
\left|F^*\right| + \norm{R^*}\varepsilon$ and using the elementary
inequality $|1 - e^z| \le |z|e^{|z|}$ we can bound
\[\left|\frac{1}{t}e^{\scal{u_0}{X_t - X_0}} \left(1 - \exp(-\int_0^{t \wedge \tau} (F^{\ast} + \scal{R^{\ast}}{X_{s-}}) ds\right)\right| \le Ce^{2l + Ct}\]
and therefore apply again dominated convergence to the right hand side of \eqref{Eq:Fstar} as  $t \to 0$. Taking the limit on both sides, we obtain
\[ F(u_0) + \scal{R(u_0)}{x} = F^{\ast} + \scal{R^{\ast}}{x}\]
leading to the desired contradiction.
\end{proof}

We conclude with a corollary that gives conditions for an affine
process to be a $D$-valued semimartingale, up to its explosion time.
Let $\tau_n = \inf \set{t \ge 0: \norm{X_t - X_0}
> n}$ and define the explosion time $\tau_\text{exp}$ as the
pointwise limit $\tau_\text{exp} = \lim_{n \to \infty} \tau_n$. Note
that $\tau_\text{exp}$ is predictable.

\begin{cor}
Let $X$ be a c\`adl\`ag affine process and suppose that the killing terms vanish, i.e.~$c = 0$ and $\gamma =
0$. Then under every $\PP^x, x \in D$ the process $X$ is a $D$-valued
semi-martingale on $[0,\tau_\text{exp})$ 
with absolutely continuous semimartingale characteristics
\begin{align*}
A_t &= \int_0^t A(X_{s-}) ds\\
B_t &= \int_0^t B(X_{s-}) ds\\
K([0,t],d\xi) &= \int_0^t \nu(X_{s-},d\xi) ds.
\end{align*}
where $A(.), B(.)$ and $\nu(.,d\xi)$ are given by \eqref{Eq:AB_nu}.
\end{cor}
\begin{proof}
In the proof of Theorem~\ref{Thm:main} we have shown that $t \mapsto
L^u_{t \wedge \tau}$, with $L_t^u$ defined in \eqref{Eq:Ldef} and
$\tau = \inf \set{t \ge 0: \norm{X_t - X_0} > r}$, is a martingale
under every $\PP^x, x \in D$ and for every $u \in \cU$. Since $r >
0$ was arbitrary, also $L^u_{t \wedge \tau_n}$ is a martingale for
every $n \in \NN$. By dominated convergence and using that $F(0) +
\scal{R(0)}{x} = c + \scal{\gamma}{x} = 0$ for all $x \in D$ we
obtain
\[\PP^x \left(X_{t \wedge \tau_\text{exp}} \neq \delta\right) = \lim_{n \to \infty} \PP \left(X_{t \wedge \tau_n} \neq \delta\right) = \E{L_{t \wedge \tau_n}^0} = 1.\]
Hence $X_t$ and $X_{t-}$ stay $\PP^x$-almost surely in $D \subset
\RR^d$ for $t \in [0,\tau_\text{exp})$. Moreover $t \mapsto L^u_t$
is a local martingale on $[0,\tau_\text{exp})$ for all $u \in \cU$.
Thus \citet[Cor.~II.2.48b]{Jacod1987} can be applied to the local
martingale $L_t^u$ with $u \in i\RR^d$ and the assertion follows.
\end{proof}

\bibliographystyle{plainnat}
\bibliography{references}

\begin{thebibliography}{18}
\providecommand{\natexlab}[1]{#1}
\providecommand{\url}[1]{\texttt{#1}}
\expandafter\ifx\csname urlstyle\endcsname\relax
  \providecommand{\doi}[1]{doi: #1}\else
  \providecommand{\doi}{doi: \begingroup \urlstyle{rm}\Url}\fi

\bibitem[Bru(1991)]{Bru1991}
M.-F. Bru.
\newblock Wishart processes.
\newblock \emph{Journal of Theoretical Probability}, 4\penalty0 (4):\penalty0
  725--751, 1991.

\bibitem[Bucchianico(1991)]{Bucchianico1991}
A.~Di Bucchianico.
\newblock {B}anach algebras, logarithms, and polynomials of convolution type.
\newblock \emph{Journal of Mathematical Analysis and Applications},
  156:\penalty0 253--273, 1991.

\bibitem[Cuchiero and Teichmann(2011)]{Cuchiero2011a}
C.~Cuchiero and J.~Teichmann.
\newblock Path properties and regularity of affine processes on general state
  spaces.
\newblock arXiv:1107.1607, 2011.

\bibitem[Cuchiero(2011)]{Cuchiero2011}
Christa Cuchiero.
\newblock \emph{Affine and Polynomial Processes}.
\newblock PhD thesis, ETH Z{\"u}rich, 2011.

\bibitem[Cuchiero et~al.(2011{\natexlab{a}})Cuchiero, Filipovic, Mayerhofer,
  and Teichmann]{Cuchiero2009}
Christa Cuchiero, Damir Filipovic, Eberhard Mayerhofer, and Josef Teichmann.
\newblock Affine processes on positive semidefinite matrices.
\newblock \emph{Annals of Applied Probability}, 21\penalty0 (2):\penalty0
  397--463, 2011{\natexlab{a}}.
\newblock To appear in {T}he {A}nnals of {A}pplied {P}robability.

\bibitem[Cuchiero et~al.(2011{\natexlab{b}})Cuchiero, Keller-Ressel,
  Mayerhofer, and Teichmann]{CKMT2011}
Christa Cuchiero, Martin Keller-Ressel, Eberhard Mayerhofer, and Josef
  Teichmann.
\newblock Affine processes on symmetric cones.
\newblock Draft, 2011{\natexlab{b}}.

\bibitem[Dawson and Li(2006)]{Dawson2006}
D.~A. Dawson and Zenghu Li.
\newblock Skew convolution semigroups and affine markov processes.
\newblock \emph{The Annals of Probability}, 34\penalty0 (3):\penalty0 1103 --
  1142, 2006.

\bibitem[Duffie et~al.(2003)Duffie, Filipovic, and Schachermayer]{Duffie2003}
D.~Duffie, D.~Filipovic, and W.~Schachermayer.
\newblock Affine processes and applications in finance.
\newblock \emph{The Annals of Applied Probability}, 13\penalty0 (3):\penalty0
  984--1053, 2003.

\bibitem[Duffie and Kan(1996)]{Duffie1996}
Darrell Duffie and Rui Kan.
\newblock A yield-factor model of interest rates.
\newblock \emph{Mathematical Finance}, 6:\penalty0 379 -- 406, 1996.

\bibitem[Faraut and Kor{\'a}nyi(1994)]{Faraut1994}
Jacques Faraut and Adam Kor{\'a}nyi.
\newblock \emph{Analysis on Symmetric Cones}.
\newblock Oxford Science Publications, 1994.

\bibitem[Jacod and Shiryaev(1987)]{Jacod1987}
J.~Jacod and A.N. Shiryaev.
\newblock \emph{Limit Theorems for Stochastic Processes}.
\newblock Springer, 1987.

\bibitem[Kallsen(2006)]{Kallsen2006}
Jan Kallsen.
\newblock A didactic note on affine stochastic volatility models.
\newblock In Y.~Kabanov, R.~Liptser, and J.~Stoyanov, editors, \emph{From
  Stochastic Calculus to Mathematical Finance}, pages 343 -- 368. Springer,
  Berlin, 2006.

\bibitem[Kawazu and Watanabe(1971)]{Kawazu1971}
Kiyoshi Kawazu and Shinzo Watanabe.
\newblock Branching processes with immigration and related limit theorems.
\newblock \emph{Theory of Probability and its Applications}, XVI\penalty0
  (1):\penalty0 36--54, 1971.

\bibitem[Keller-Ressel et~al.(2011)Keller-Ressel, Schachermayer, and
  Teichmann]{KST2011}
M.~Keller-Ressel, W.~Schachermayer, and J.~Teichmann.
\newblock Affine processes are regular.
\newblock \emph{Probability Theory and Related Fields}, 151\penalty0
  (3-4):\penalty0 591--611, 2011.
\newblock arXiv:1105.0632.

\bibitem[Keller-Ressel(2008)]{K2008b}
Martin Keller-Ressel.
\newblock \emph{Affine Processes -- Contributions to Theory and Applications}.
\newblock PhD thesis, TU Wien, 2008.

\bibitem[Montgomery and Zippin(1955)]{Montgomery1955}
Deane Montgomery and Leo Zippin.
\newblock \emph{Topological Transformation Groups}.
\newblock Interscience Publishers, Inc., 1955.

\bibitem[Sato(1999)]{Sato1999}
Ken-Iti Sato.
\newblock \emph{{L}{\'e}vy processes and infinitely divisible distributions}.
\newblock Cambridge University Press, 1999.

\bibitem[Spreij and Veerman(2010)]{Spreij2010}
P.~Spreij and E.~Veerman.
\newblock Affine diffusions with non-canonical state space.
\newblock arXiv:1004.0429, 2010.

\end{thebibliography}

\end{document}